\documentclass[a4, 11pt] {amsart} 
\usepackage{amssymb,times, amscd,amsmath,amsthm, xypic}
\usepackage{geometry}
\usepackage{mathrsfs}
\usepackage{amsmath}
\usepackage{amssymb}
\usepackage[all]{xy}
\usepackage[dvips]{graphicx, color}

\numberwithin{equation}{section}
\newtheorem{thm}[equation]{Theorem}
\newtheorem{pro}[equation]{Proposition}

\newtheorem{cor}[equation]{Corollary}

\newtheorem{lem}[equation]{Lemma}
\theoremstyle{definition}
\newtheorem{ex}[equation]{Example}
\newtheorem{defn}[equation]{Definition}
\newtheorem{ob}[equation]{Observation}
\newtheorem{rem}[equation]{Remark}

\def\alp{\alpha}

\def\la{\lambda}

\def\ga{\gamma}

\def\bZ{\mathbb Z}

\def \bQ{\mathbb Q}

\def\op{\operatorname}

\begin{document}

\title{Poset-stratified space structures of homotopy sets} 

\author{ Toshihiro Yamaguchi \ \ and \ \ Shoji Yokur$\mbox{a}^*$}
\thanks{
\\
(*) Partially supported by JSPS KAKENHI Grant Numbers 16H03936}

\date{}

\address{Faculty of Education, Kochi University, 2-5-1,Kochi, 780-8520, Japan} 
\email{tyamag@kochi-u.ac.jp}

\address{Department of Mathematics  and Computer Science, Graduate School of Science and Engineering, Kagoshima University, 1-21-35 Korimoto, Kagoshima, 890-0065, Japan
}

\email{yokura@sci.kagoshima-u.ac.jp}

\maketitle 

\begin{abstract} A poset-stratified space is a pair $(S, S \xrightarrow \pi P)$ of a topological space $S$ and a continuous map $\pi: S \to P$ with a poset $P$ considered as a topological space with its associated Alexandroff topology. In this paper we show that one can impose such a poset-stratified space structure on the homotopy set $[X, Y]$ of homotopy classes of continuous maps by considering a canonical but non-trivial order (preorder) on it, namely we can capture the homotopy set $[X, Y]$ as an object of the category of poset-stratified spaces. The order we consider is related to the notion of \emph{dependence of maps} (by Karol Borsuk). Furthermore via homology and cohomology the homotopy set $[X,Y]$ can have other poset-stratified space structures. In the cohomology case, we get some results which are equivalent to the notion of \emph{dependence of cohomology classes} (by Ren\'e Thom) and we can show that the set of isomorphism classes of complex vector bundles can be captured as a poset-stratified space via the poset of the subrings consisting of all the characteristic classes.  We also show that some invariants such as Gottlieb groups and Lusternik--Schnirelmann category of a map give poset-stratified space structures to the homotopy set $[X,Y]$.
\end{abstract}

%%%%%%%%%%%%%%
\section{Introduction}
The homotopy set $[X,Y]$ is the set of homotopy classes of continuous maps from a topological space $X$ to another one $Y$.
In our previous work \cite{YY} we consider a preorder on the homotopy set $[X,Y]$ using the action of the self-homotopy equivalences $\mathcal E[X]$ of $X$ and the self-homotopy equivalences $\mathcal E[Y]$ of $Y$ on $[X,Y]$. Using such a preordered set (proset), we consider some classification of Hurewicz fibrations.

In this paper we consider another preorder on $[X,Y]$ via the action of monoids $[X,X]$ and $[Y,Y]$ on $[X,Y]$, instead of $\mathcal E[X]$ and $\mathcal E[Y]$. Here we note that a homotopy class $[f] \in \mathcal E[X]$ has its inverse $[f]^{-1} \in \mathcal E[X]$, but a homotopy class $[f] \in [X,X]$ does not always have an inverse $[f]^{-1} \in [X,X]$, which is a substantial difference between $[X,X]$ and $\mathcal E[X]$. For example, we consider the following order:
$$[f] \leqq_R [g] \Longleftrightarrow \exists [s] \in [X, X] \,\,\, \text{such that $[f] = [g] \circ [s]$},$$
i.e. the following diagram commutes up to homotopy ($f \sim g \circ s$):
$$\xymatrix
{
X \ar [d]_s \ar [r]^{f} & Y\\
X \ar [ur] _{g} & }.
$$
This is a preorder.
Then we consider the following equivalence relation $\sim_R$ using this preorder $\leqq_R$:
$$\text{$[f] \sim_{R} [g]  \Longleftrightarrow  [f] \leqq_R [g]$ and $[g] \leqq_R [f],$} $$
namely,
$$\text{$\exists [s_1], [s_2] \in \mathcal \mathcal [X, X]$ such that $
[f] =[g] \circ [s_1], [g]=[f] \circ [s_2]$},$$
i.e. the following diagram commutes up to homotopy:
$$\xymatrix
{
X \ar@<0.8ex> [d]^{s_1} \ar [r]^{f} & Y\\
X \ar@<0.8ex>[u]^{s_2} \ar [ur] _{g} & }.
$$

The relation $\sim _R$ is an equivalence relation, called \emph{right equivalence relation} and the set of equivalence classes shall be denoted by
$[X,Y]_{ R}  := [X,Y]/ \sim_{ R}.$ The equivalence class of $[f]$ is denoted by $[f]_R$. We define the order $\leqq'_R$ on $[X,Y]_{ R}$ by $[f]_R \leqq'_R [g]_R \Longleftrightarrow [f] \leqq_R [g]$. This order $\leqq'_R$ 
is well-defined and becomes a partial order. Thus the canonical map $\pi_R: ([X,Y], \leqq_R) \to ([X,Y]_R, \leqq'_R)$ is a monotone (order-preserving) map from a proset to a poset. If we consider the Alexandroff topologies $\tau_{\leqq_R}$ on the source $([X,Y], \leqq_R)$ and $\tau_{\leqq'_R}$ on the target $([X,Y]_R, \leqq'_R)$, this in turn gives us a continuous map $\pi_R:([X,Y], \tau_{\leqq_R}) \to ([X,Y]_R, \tau_{\leqq'_R})$.  In other words, this is \emph{a continuous map from a topological space $([X,Y], \tau_{\leqq_R})$ to a poset 
$([X,Y]_R, \leqq'_R)$ which is considered as a topological space $([X,Y]_R, \tau_{\leqq'_R})$ with the Alexandroff topology}. Such a map is called \emph{a poset-stratified space} in modern terminology (e.g., see \cite{Lurie}). 
\begin{rem} In the case when we consider the self-homotopy equivalences $\mathcal E[X]$ of $X$, instead of the monoid $[X,X]$, since each element $[s] \in \mathcal E[X]$ has its inverse $[s]^{-1} \in \mathcal E[X]$ (more precisely, $\exists s':X \to X$ such that $s \circ s' \sim id_X$ and $s' \circ s \sim id_X$, thus $[s]^{-1} = [s']$), the above equivalence relation $\sim_R$ is replaced simply by the following equivalence relation $\sim_{\mathcal ER}$:
$$[f] \sim_{\mathcal ER} [g] \Longleftrightarrow \exists [s] \in \mathcal E[X] \,\,\, \text{such that $[f] = [g] \circ [s]$},$$
i.e. the following diagram commutes up to homotopy ($f \sim g \circ s$):
$$\xymatrix
{
X \ar [d]_s \ar [r]^{f} & Y\\
X \ar [ur] _{g} & }.
$$
Because $f \sim g \circ s$ automatically implies that $g \sim f \circ s'$. On the set $[X,Y]_{\mathcal ER}$ of equivalence classes, as in the case of $[X,Y]_R$, one can define the following order for $[f]_{\mathcal ER}, [g]_{\mathcal ER} \in [X,Y]_{\mathcal ER}$
$$\text{$[f]_{\mathcal ER}  \leqq_R [g]_{\mathcal ER} \Longleftrightarrow  \exists s \in [X, X]$ (not $\exists  s \in \mathcal E[X]$) \, \, such that $[f] = [g] \circ [s]$.} \, \,  $$
Here we emphasize that this order is \emph{not necessarily} a partial order, but that the above order $\leqq'_R$ on $[X,Y]_{ R}$ defined by $[f]_R \leqq'_R [g]_R \Longleftrightarrow [f] \leqq_R [g]$ \emph{is} a partial order, because of the equivalence relation $[f] \sim_{R} [g]$ defined by $\exists [s_1], [s_2] \in \mathcal \mathcal [X, X]$ such that $
[f] =[g] \circ [s_1], [g]=[f] \circ [s_2]$. One could think of such a pair $([s_1], [s_2])$ as a ``mock" self-homotopy equivalence of $X$ with respect to the pair $(f, g)$.

\end{rem}

Similarly we consider the preorder 
$$\text{$[f] \leqq_L [g] \Longleftrightarrow \exists [t] \in [Y,Y]$ such that $[f] = [t] \circ [g]$,}$$
i.e. the following diagram commutes up to homotopy ($f \sim t \circ g$ \footnote{As remarked later, in a different context Karol Borsuk \cite{Bor1, Bor2} considered such a relation when he characterized his definition of \emph{ $f:X \to Y$ depending on $g:X \to Y$}.}):
$$\xymatrix
{
X \ar[dr]_g \ar [r]^{f} & Y\\
& Y. \ar[u]_t}
$$
Then we consider the following equivalence relation $\sim_L$ using this preorder $\leqq_L$:

$$\text{$[f] \sim_L [g]  \Longleftrightarrow  [f] \leqq_L [g]$ and $[g] \leqq_L [f],$}$$
i.e., $\exists [t_1], [t_2] \in \mathcal \mathcal [Y, Y]$ such that $
[f] =[t_1] \circ [g], [g]= [t_2] \circ [f]$, i.e. the following diagram commutes up to homotopy:
$$\xymatrix
{
X \ar[dr]_g \ar [r]^{f} & Y \ar@<0.8ex> [d]^{t_2} \\
& Y. \ar@<0.8ex>[u]^{t_1}
}
$$

The equivalence relation $\sim _L$ is called \emph{left equivalence relation} and the set of equivalence classes shall be denoted 
$[X,Y]_{ L}  := [X,Y]/ \sim_{ L}.$ As in the case of $\leqq_R$ and $\leqq'_R$, the canonical map $\pi_L: ([X,Y], \leqq_L) \to ([X,Y]_L, \leqq'_L)$ is a monotone map from a proset to a poset.

These poset-stratified spaces can be captured as functors from the homotopy category of topological spaces to the category of poset-stratified spaces as follows:
\begin{thm} Let $h\mathcal Top$ be the homotopy category of topological spaces. 
\begin{enumerate}
\item For any object $S \in Obj(h\mathcal Top)$, we have an associated covariant functor
$\frak {st}^S_*: h\mathcal Top \to \mathcal Strat$
such that
\begin{enumerate}
\item for each object $Y \in Obj(h\mathcal Top)$, 
$$\frak {st}^S_*(X) := \left ( ([S,X], \tau_{\leqq_R}), ([S,X], \tau_{\leqq_R}) \xrightarrow {\pi_R} ([S,X]_R, \leqq'_R) \right)$$
\item for a morphism $[f] \in [X,Y]$ , $\frak {st}^S_*([f])$ is the following commutative diagram:
$$\xymatrix
{
([S,X], \tau_{\leqq_R}) \ar [r]^{\pi_R} \ar [d]_{f_*} & ([S,X]_R, \leqq'_R)  \ar [d]^{f_*} \\
([S,Y], \tau_{\leqq_R})  \ar[r]_{\pi_R} & ([S,Y]_R, \leqq'_R) \\
}
$$
\end{enumerate}
\item For any object $T \in Obj(h\mathcal Top)$, we have an associated contravariant functor
$\frak {st}_T^*: h\mathcal Top \to \mathcal Strat$
such that
\begin{enumerate}
\item for each object $X \in Obj(h\mathcal Top)$, 
$$\frak {st}_T^*(X) := \left ( ([X,T], \tau_{\leqq_L}), ([X,T], \tau_{\leqq_L}) \xrightarrow {\pi_L} ([X,T]_L, \leqq'_L) \right)$$
\item for a morphism $[f] \in [X,Y]$  , $\frak {st}_T^*([f])$ is the following commutative diagram:
$$\xymatrix
{
([Y,T], \tau_{\leqq_L}) \ar [r]^{\pi_L} \ar [d]_{f^*} & ([Y,T]_L, \leqq'_L)  \ar [d]^{f^*} \\
([X,T], \tau_{\leqq_L})  \ar[r]_{\pi_L} & ([X,T]_L, \leqq'_L) \\
}
$$
\end{enumerate}
\end{enumerate}
\end{thm}

In Example \ref{ex1} we see an example such that the homotopy sets are the same:$[S,X]=[S,Y]$, but their poset-stratified space structures are different: $\frak {st}^S_*(X) \not = \frak {st}^S_*(Y).$

By considering homology and cohomolgy, and homotopy and cohomotopy, we can get other more ``algebraic" or ``geometric" poset-stratified space structures on the homotopy set.
For example, consider the homotopy set $[S^1, S^1] = \mathbb Z$. Then the preorder $a \leqq_R b$ is by our definition nothing but $\exists s \in \mathbb Z$ such that $a = b\cdot s$, i.e., $b$ divides $a$, $b|a$. For an integer $n \in \mathbb Z= [S^1, S^1]$, i.e., $n$ is the homotopy class of the map $z^n: S^1 \to S^1$ and consider $(z^n)_*:H_1(S^1) \to H_1(S^1)$ or $(z^n)_*:\pi_1(S^1) \to \pi_1(S^1)$, which gives us the homomorphism $\times n: \mathbb Z \to \mathbb Z$. Then the image $\op{Im}(\times n) = (n) =\{ kn \, |\, k \in \mathbb Z\}$ is the subgroup %of 
generated by the integer $n$. The set $\mathcal Sub(\mathbb Z)$ of all the subgroups of $\mathbb Z$ is $\{ (n) \, | \, n \in \mathbb Z \}$ and the order $(a) \leqq (b)$ defined by the inclusion $(a) \subset (b)$, which means that $\exists s \in \mathbb Z$ such that $a = b\cdot s$, thus $b|a$.
Thus the map $\op{Im}_{H_1}:([S^1, S^1], \leqq_R) =(\mathbb Z, \leqq_R) \to (\mathcal Sub(\mathbb Z), \leqq)$ defined by $\op{Im}_{H_1}(n) = \op{Im}((z^n)_*) =(n)$
is a monotone map.

In the case of $([X,Y], \leqq_L)$ we consider the cohomology theory $H^*(-; \mathbb Z)$ and we get a canonical monotone map 
$\op{Im}_{H^*}:([X,Y], \leqq_L) \to (\mathcal Sub(H^*(X)), \leqq)$, which is 
 defined by $\op{Im}_{H^*}([f]):= \op{Im}(f^*:H^*(Y) \to H^*(X)) =f^*(H^*(Y))$. Here $\mathcal Sub(H^*(X))$ is the set of all the subgroups of $H^*(X)$ and the order $S_1 \leqq S_2$ for subgroups $S_1, S_2 \in \mathcal Sub(H^*(X))$ is the usual inclusion $S_1 \subset S_2$. This \emph{monotone} map $\op{Im}_{H^*}:([X,Y], \leqq_L) \to (\mathcal Sub(H^*(X)), \leqq)$ has a connection with R. Thom's notion of \emph{dependence of cohomology classes} \cite{Thom}. Indeed, %we 
let us consider $Y = K(\mathbb Z, p)$ the Eilenberg-Maclane space, %and 
then we have $\op{Im}_{H^*}:([X,K(\mathbb Z, p)], \leqq_L) \to (\mathcal Sub(H^*(X)), \leqq)$. Since $H^p(X, \mathbb Z) =[X, K(\mathbb Z, p)]$, let $f_{\alpha}:X \to K(\mathbb Z, p)$ be a map whose homotopy class $[f_{\alpha}]$ corresponds to the cohomology class $\alpha \in H^p(X, \mathbb Z)$.  Let $\beta \in H^p(X, \mathbb Z)$ be another cohomology class, thus we consider the corresponding homotopy class $[f_{\beta}]$. Let $[f_{\beta}] \leqq_L [f_{\alpha}]$, i.e. $\exists [t] \in [K(\mathbb Z, p),K(\mathbb Z, p)]$ such that $[f_{\beta}] = [t] \circ [f_{\alpha}]$ ($f_{\beta} \sim t \circ f_{\alpha}$), which implies that $\op{Im}(f_{\beta}^*) =f_{\beta}^*(H^*(K(\mathbb Z, p))) \subset \op{Im}(f_{\alpha}^*) =f_{\alpha}^*(H^*(K(\mathbb Z, p)))$. In particular, $\beta \in f_{\beta}^*(H^*(K(\mathbb Z, p)))$, thus $\beta \in f_{\alpha}^*(H^*(K(\mathbb Z, p)))$, which implies by Thom's definition of \emph{dependence of cohomology classes} \cite{Thom} (also see \cite{Hil1}) that \emph{the cohomology class $\beta$ depends on the cohomology class $\alpha$}. Thus the upshot is that our $[f_{\beta}] \leqq_L [f_{\alpha}]$, namely, that \emph{$f_{\beta}$ depends on $f_{\alpha}$} (using Borsuk's definition of \emph{dependence of maps}) implies that \emph{$\beta$ depends on $\alpha$}. 
 
If we consider $Y= G_n(\mathbb C^{\infty})$ the infinite Grassmann of $n$-dimensional planes in $\mathbb C^{\infty}$ for $\op{Im}_{H^*}:([X,Y], \leqq_L) \to (\mathcal Sub(H^*(X)), \leqq)$, then we get a natural ``order" among the isomorphism classes of complex vector bundles. Indeed, if we denote the set of isomorphism classes of complex vector bundles of rank $n$, then we know that 
$\op{Vect}_n(X) \cong [X, G_n(\mathbb C^{\infty})]$, which is by the correspondence $[E] \leftrightarrow [f_E]$, where $f_E:X \to G_n(\mathbb C^{\infty})$ is a classifying map of $E$, i.e., $E = f_E^* \gamma^n,$ 
where $\gamma^n$ is the universal complex vector bundle of rank $n$ over $G_n(\mathbb C^{\infty})$. 
By the isomorphism $\op{Vect}_n(X) \cong [X, G_n(\mathbb C^{\infty})]$ we can consider the peorder on $\op{Vect}_n(X)$: $[E] \leqq_L [F] \Longleftrightarrow [f_E] \leqq_L [f_F],$
where $f_E, f_F:X \to G_n(\mathbb C^{\infty})$ are respectively the classifying maps of $E$ and $F$.
Then we have the following well-defined monotone (order-preserving) map:
$$ \op{Im}_{H^*}: (\op{Vect}_n(X), \leqq_L) \to (\mathcal Sub(H^*(X; \mathbb Z)), \leqq )$$
defined by 
$ \op{Im}_{H^*}([E]):= \op{Im} \left (f_E^*:H^*(G_n(\mathbb C^{\infty});\mathbb Z) \to H^*(X; \mathbb Z) \right ).$
By the definition of characteristic classes, \emph{ $\op{Im}\left (f_E^*:H^*(G_n(\mathbb C^{\infty});\mathbb Z) \to H^*(X; \mathbb Z) \right )$ is the subring consisting of \underline{all} the characteristic classes of $E$}, denoted by $\op{Char}(E)$. Therefore we have
$[E] \leqq_L[F] \Longrightarrow \op{Char}(E) \subseteq \op{Char}(F).$
We also get that  $[E] \sim_L[F] \Longrightarrow \op{Char}(E) = \op{Char}(F).$

We also show that the Gottlieb groups and Lusternik-Schnirelmann category of a map give %the homotopy se
poset-stratified space structures to homotopy sets.

%%%%%%%%%%%%%%%%%%
\section{Preliminaries }

In this section we give some preliminaries for later use.

A preorder on a set $P$ is a relation $\leqq$ which is reflexive ($a \leqq a$) and transitive ($a \leqq b, b \leqq c \Longrightarrow a \leqq c$).
A set $(P, \leqq)$ equipped with a preorder $\leqq$ is called a \emph{proset} (preordered set).
If a preorder $\leqq $ is anti-symmetric ($a \leqq b, b \leqq a \Longrightarrow a =b$), then it is called a partial order and a set with a partial order is called a \emph{poset} (partially ordered set).

\begin{defn}[Alexandroff topology \cite{A1}]  Let $X$ be a topological space. If \emph{the intersection of any family} of open sets is open or equivalently the union of any family of closed sets is closed, then the topology is called an \emph{Alexandorff topology} and the space is called an \emph{Alexandroff space} %(cf. \cite{Ar, Sp})
.
\end{defn}

For Alexandroff topology or spaces, e.g., see \cite{A1}, \cite{A2}, \cite{Ar}, \cite[\S 4.2.1 Alexandroff Topology]{Cur}, \cite{Sp}, \cite[Appendix A  Pre-orders and spaces]{Woolf}.

Note that any finite topological space, i.e. a finite set with a topology, is clearly an Alexandroff space. (For finite topological spaces, e.g., see \cite{B, May1, May2, May3, Mc1, Mc2, St}.)

Given a proset $(X, \leqq)$, we define $U \subset X$ to be an open set if and only if 
$x \in U, \, x \leqq y \Rightarrow y \in U$, i.e. if and only if $U$ is \emph{closed upwards}\footnote{The Alexandroff topology is sometimes considered by defining an open set to be \emph{closed downwards} instead of closed upwards, e.g., see \cite{Ar}, \cite{B}, \cite{May1} and \cite{Sp}. When stratification theory or poset-stratified spaces are considered as in the above cited references \cite{Cur} and \cite{Woolf}, upward closedness is used in defining Alexandroff topology (e.g., see \cite[Definition A.5.1]{Lurie} and \cite[Definition 2.1 ]{Tam} as well).}.
In other words, if we let $U_x := \{y \in X \, | \, x \leqq y \}$, then $\{U_x \, | \, x \in X\}$ is the base for the topology. This topology is denoted by $\tau_{\leqq}$. 

\begin{lem} For a proset $(X, \leqq)$, the topological space $(X, \tau_{\leqq})$ is an Alexandroff space.
\end{lem}
Because of this, the topology $\tau_{\leqq}$ is called the Alexandroff topology (associated to the preorder).%We remark that a 
\begin{ob} A subset $F$ is a closed set in the topology $\tau_{\leqq}$ if and only if
$x \in F, \, y \leqq x \Rightarrow y \in F.$ 
\end{ob}
From this observation we can see that if $P$ is a poset, \emph{not a proset}, for any point $x \in P$,
$\{x\} = \{y \in P \, | \, x \leqq y\} \cap \{y \in P \, | \, y \leqq x \}.$ In other words, in the associated Alexandroff topology $\tau_{\leqq}$ any singleton $\{x\}$ is a locally closed set, i.e., the intersection of a closed set and an open set. Note that for example, for a two-point proset $(\{a, b\}, \leqq)$ with the preorder $\leqq$ defined by $a\leqq b, b \leqq a$, the above observation does not hold.

If we let $\mathcal Proset$ be the category of prosets and monotone (order-preserving) functions of prosets and $\mathcal Alex$ be the category of Alexandroff spaces and continuous maps, then we have a covariant functor $\mathcal T: \mathcal Proset \to \mathcal Alex.$

Conversely, for a topological space $(X, \tau)$, we define the following order, called \emph{specialization order}, on $X$:
$\text {$x \leqq_{\tau} y \Leftrightarrow x \in \overline {\{y\}}$}.$
Certainly this is a preorder, but not necessarily a partial order. (For example, for any indiscrete topological space having more than or equal to two points, it is never a partial porder.) If $f:(X, \tau_1) \to (Y, \tau_2)$ is a continuous map, then $f:(X, \leqq_{\tau_1}) \to (Y, \leqq_{\tau_2}) $ is a monotone function.
Therefore we have a covariant functor $\mathcal P: \mathcal Top \to \mathcal Proset.$
We have that for any proset $(X, \leqq)$, 
$\left( \mathcal P \circ \mathcal T \right) ((X, \leqq )) = (X, \leqq )$, i.e., $\mathcal P \circ \mathcal T = Id_{\mathcal Proset}.$
However, in general, for a topological space $(X, \tau)$ we have
$\left( \mathcal T \circ \mathcal P \right) ((X, \tau)) \not = (X, \tau)$, i.e., $\mathcal T \circ \mathcal P \not = Id_{\mathcal Top}$.
The reason is simple: $\left( \mathcal T \circ \mathcal P \right) ((X, \tau))$ is always an Alexandroff space, even if the original space $(X, \tau)$ is not an Alexandroff space, namely the topology of $\left( \mathcal T \circ \mathcal P \right) ((X, \tau))$ is stronger that the original topology $\tau$. 
However, if we restrict the covariant functor $\mathcal P: \mathcal Top \to \mathcal Proset$ to the subcategory $\mathcal Alex$ of Alexandroff spaces, then we have 
$\left( \mathcal T \circ \mathcal P \right) ((X, \tau)) = (X, \tau)$, i.e., $\mathcal T \circ \mathcal P = Id_{\mathcal Alex}$.
Therefore we have that 
$\mathcal P \circ \mathcal T = Id_{\mathcal Proset}, \mathcal T \circ \mathcal P = Id_{\mathcal Alex}.$
Thus Alexandroff spaces and prosets are equivalent.

For a proset $(P, \leqq)$, we can consider the reversed order, denoted $\leqq ^{op}$, by
$a \leqq^{op} b \Leftrightarrow b \leqq a.$
Here we note that the Alexandroff topologies associated to the two prosets $(P, \leqq)$ and $(P, \leqq^{op})$ of the same set $P$ are different. 
%In the above the Alexandroff topology is defined using upward-set (or ``upset"). However one can define another Alexandroff topology using down-set. For example, in \cite{B} J. A. Barmark considers the Alexandroff topology using down-set, namely the basic open set is $U_x =\{ y \, | \, y \leqq x\}$. 

%%%%%%%%%%%%%%%%%%%%
A stratification of a topological space (which can be the underlying topological
space of a much finer object such as a complex algebraic variety, a complex analytic space) is a special kind of decomposition with certain extra conditions. 
It seems that there is no fixed or standard definition of \emph{stratification} and there are several ones depending on the objects to study, such as topologicaly stratified spaces and Thom--Whitney stratified spaces. In \cite{Tam} D. Tamaki gives a nice review of several stratifications available in mathematics.

Here is one definition of stratification:
\begin{defn}\label{stratification}
Let $X$ be a topological space. If a family $\{e_{\la}\}_{\la \in \Lambda}$ of subsets of $X$ satisfies the following conditions, then $\{e_{\la}\}_{\la \in \Lambda}$  is called \emph{a stratification} of $X$.
\begin{enumerate}
\item $e_{\la} \cap e_{\mu} = \emptyset$ if $\la \not = \mu.$
\item $X = \bigcup_{\la} e_{\la}$.
\item (locally closed set) Each $e_{\la}$ is a locally closed set.
\item (frontier condition) $e_{\la} \cap \overline{e_{\mu}} \not = \emptyset  \Longrightarrow e_{\la} \subset \overline{e_{\mu}}.$
\end{enumerate}
\end{defn}

%%%%%%%%%%%%%%%%%%%%%%%%%
Just a decomposition requires only (1) and (2). Given a decomposition $\mathcal D$ of $X$, we have the quotient map
$\pi_{\mathcal D}: X \to X/\mathcal D,$
which means that one considers each piece $e_{\lambda}$ as a point. Then we can identify $X/\mathcal D = \Lambda$. We consider the quotient topology, denoted $\tau_{\pi_{\mathcal D}}$, on the target $\Lambda$, i.e., the finest or strongest topology on $\Lambda$ such that the quotient map $\pi_{\mathcal D}: X \to X/\mathcal D = \Lambda$ becomes a continuous map. Suppose that the quotient topology $\tau_{\pi_{\mathcal D}}$ is an Alexandroff topology, which is the case when the decomposition $\mathcal D$ is finite, i.e., $\Lambda$ is a finite set. Then we get the preorder $\leqq_{\tau_{\pi_{\mathcal D}}}$.
If $\leqq_{\tau_{\pi_{\mathcal D}}}$ is a partial order, then each piece $e_{\lambda} = \pi_{\mathcal D}^{-1}(\lambda)$ has to be locally closed, because each singleton 
$\{\lambda\}$ is a locally closed set as observed above. At the moment we do not know if the converse holds, i.e., whether each piece $e_{\lambda}$ being locally closed  implies that $\leqq_{\tau_{\pi_{\mathcal D}}}$ is a partial order.

As to the preorder on $\Lambda$, we can define it using the above ``frontier condition" %as follows:
by
$\la \leqq^* \mu \Longleftrightarrow e_{\la} \subset \overline {e_{\mu}}.$
Then one can see that each piece $e_{\lambda}$ being locally closed implies that the above preorder $\leqq^*$ is in fact a partial order. Furthermore the quotient map $\pi_{\mathcal D}:X \to X/\mathcal D=\Lambda$ is a continuos map with the Alexandroff topology $\tau_{\leqq^*}$ associated to the order $\leqq^*$ if and only if the Alexandroff topology $\tau_{\leqq^*}$ is equal to the quotient topology. In other words, if the decomposition space $X/\mathcal D =\Lambda$ with the quotient topology is an Alexandroff space, then the order $\leqq^*$ is the same as $\leqq_{\tau_{\pi_{\mathcal D}}}$, i.e., $\la \leqq_{\tau_{\pi_{\mathcal D}}} \mu \Longleftrightarrow e_{\la} \subset \overline {e_{\mu}}$.

Such \emph{a continuous map from a topological space to a poset considered as a topological space with the Alexandroff topology} has been studied in recent papers (e.g., \cite{AFT, Cur, Lurie, Tam, Yo}, etc.)

%%%%%%%%%%%%%%%%%%
\begin{defn} Let $P$ be a poset. A \emph{poset-stratified space $S$} over the poset $P$ is a pair $(S, S \xrightarrow \pi P)$ of a topological space $S$ and a continuous map $\pi: S \to P$ where $P$ is considered as the associated Alexandroff space. 
\end{defn}

%%%%%%%%%%%%%%%%%%
\begin{rem} The notion of poset-stratified space is due to Jacob Lurie \cite{Lurie}. For a poset-stratified space $(S, S \xrightarrow \pi P)$, $S$ is the underlying topological space and $\pi: S \to P$ is considered  as \emph{a structure of poset-stratification}. If the context is clear, then we just write a poset-stratified space $S$, just like writing a topological space $S$ without referring to which topology to be considered on it. 
\end{rem}

The category of poset-stratified spaces is denoted by $\mathcal Strat$. The objects are pairs $(S, S \xrightarrow \pi P)$ of a topological space $S$ and a continuous map $\pi: S \to P$ from the space $S$ to a poset $P$ with the Alexandroff topology associated to the poset $P$. Given two poset-stratified spaces $(S, S \xrightarrow {\pi} P)$ and $(S', S' \xrightarrow {\pi'} P)$, a morphism from $(S, S \xrightarrow \pi P)$ to $(S', S' \xrightarrow {\pi'} P')$ is a pair of a continuous map $f:S \to S'$  and a monotone map $q:P \to P'$ (i.e., for $a\leqq b$ in $P$ we have $q(a) \leqq q(b)$ in $P'$, thus it is a continuous map for the associated Alexandroff sapces) such that the following diagram commutes:
$$
\xymatrix{
S \ar [r]^{\pi} \ar [d]_{f} & P \ar [d]^{q} \\
S' \ar[r]_{\pi'} & P'. 
}
$$

%%%%%%%%%%%%%%%%%%%%%%%
\section{A poset-stratified space structure of $[X,Y]$}

\begin{lem} On the homotopy set $[X,Y]$ we define the following orders, which are preorders.
\begin{enumerate}
\item
$[f] \leqq_R [g] \Longleftrightarrow \exists [s] \in [X, X]$ such that $[f] = [g] \circ [s]$,
i.e. the following diagram commutes up to homotopy:
$$\xymatrix
{
X \ar [d]_s \ar [r]^{f} & Y\\
X \ar [ur] _{g} & }.
$$
\item
$[f] \leqq_L [g] \Longleftrightarrow \exists [t] \in [Y,Y]$ such that $[f] = [t]\circ [g]$,
i.e. the following diagram commutes up to homotopy:
$$\xymatrix
{
X \ar[r]^f \ar [dr]_g & Y\\
& Y. \ar[u]_t }
$$
\item
$[f] \leqq_{LR} [g] \Longleftrightarrow \exists [s] \in [X,X] , \, \exists [t] \in [Y,Y]$ such that $[f] = [t] \circ [g] \circ [s]$, 
i.e. the following diagram commutes up to homotopy:
$$\xymatrix
{
X \ar[d]^{s} \ar [r]^f & Y\\
X \ar [r]_g & Y \ar[u]^{t}
}
$$
\end{enumerate}
\end{lem}

%%%%%%%%%%%%
\begin{lem}\label{modified version} On the homotopy set $[X,Y]$ we define the following relations.
\begin{enumerate}
\item {\bf \emph{right equivalence relation}}: $[f] \sim_{R} [g]  \Longleftrightarrow  [f] \leqq_R [g]$ and $[g] \leqq_R [f],$
i.e., $\exists \, \, [s_1], [s_2] \in \mathcal \mathcal [X, X] \,\,\,  \text{such that} \, \, 
[f] =[g] \circ [s_1], \, \,  [g]=[f] \circ [s_2]$, i.e. the following diagram commutes up to homotopy:
$$\xymatrix
{
X \ar@<0.8ex> [d]^{s_1} \ar [r]^{f} & Y\\
X \ar@<0.8ex>[u]^{s_2} \ar [ur] _{g} & }.
$$
The relation $\sim _R$ is an equivalence relation and the set of equivalence classes shall be denoted by
$$[X,Y]_{ R}  := [X,Y]/ \sim_{ R} .$$
The equivalence class of $[f]$ is denoted by $[f]_R$.
\item {\bf \emph{left equivalence relation}}:
$[f] \sim_{ L}  [g] \Longleftrightarrow  [f] \leqq_L [g]$ and $[g] \leqq_L [f],$
i.e., $\exists \, \, [t_1], \, [t_2] \in [Y, Y] \,\,\,  \text{such that} \, \, [f] = [t_1] \circ [g], \, \, [g] =[t_2] \circ [f] , $
i.e. the following diagram commutes up to homotopy:
$$\xymatrix
{
X  \ar[r]^{f} \ar[dr]_{g} &Y \ar@<0.8ex>[d]^{t_2}  \\
& Y \ar@<0.8ex> [u]^{t_1}& 
}
$$
The relation $\sim_{ L}$ is an equivalence relation and the set of equivalence classes shall be denoted by
$$[X,Y]_{ L}  := [X,Y]/ \sim_{ L} .$$
The equivalence class of $[f]$ is denoted by $[f]_L$.
\item {\bf \emph{left-right equivalence relation}}
$[f] \sim_{ LR} [g] \Longleftrightarrow  [f] \leqq_{LR} [g]$ and $[g] \leqq_{LR} [f],$
i.e., $\exists \, \, [s_1], [s_2]  \in [X,X], \,\, \exists \,\, [t_1], [t_2] \in [Y, Y] \,\,\,  \text{such that}$
$$
[f] = [t_1] \circ [g] \circ [s_1] , \, \, \, [g] = [t_2] \circ [f] \circ [s_2] $$
i.e. the following diagram commutes up to homotopy:
$$\xymatrix
{
X \ar@<0.8ex>[d]^{s_1} \ar [r]^f & Y\ar@<0.8ex>[d]^{t_2}\\
X \ar@<0.8ex> [u]^{s_2}  \ar [r]_g & Y \ar@<0.8ex>[u]^{t_1}
}
$$
The relation $\sim_{ LR}$ is an equivalence relation and the set of equivalence classes shall be denoted by
$$[X,Y]_{ LR}  := [X,Y]/ \sim_{ LR} .$$
The equivalence class of $[f]$ is denoted by $[f]_{LR}$.
\end{enumerate}
\end{lem}

%%%%%%%%%%%%%%Jim Stasheff 's comments%%%%%%%%
\begin{rem} As to the above relation $[f]\leqq_L [g]$, Jim Stasheff (private communication) informed us of Karol Borsuk's papers \cite{Bor1, Bor2} and Peter Hilton's paper \cite{Hil1} (cf. \cite{Hil2, Hil3}).
K. Borsuk introduced \emph {dependence of maps}: $f: X \to Y$ is said to \emph{depend on $g:X \to Y$} if whenever $g$ is extended to $X' \supset X$, so is $f$. He gave an alternative naming for this notion: \emph{$f$ is a \emph{multiple} of $g$} or \emph{$g$ is a divisor of $f$}. It turned out that this naming was correct, because Borsuk proved that \emph{$f$ depends on $g$ if and only if there exists a map $t: Y \to Y$ such that $f \sim t \circ g$}, i.e., $[f] \leqq_L[g]$ in our notation.  Furthermore Borsuk defined two maps $f$ and $g$ to be \emph{conjugate} if they depend on each other, i.e., $[f] \sim_L [g]$ in our notatin. 
Dually, $f:X \to Y$ is said to \emph{co-depend on} $g:X \to Y$ if whenever $g$ lifts to the total space $E$ of a fibration over $Y$, so does $g$. Then the dual of the above Borsuk's result is that \emph{ $f$ co-depends on $g$ if and only if there exists a map $s: X \to X$ such that $f \sim g \circ s$}, i.e., $[f] \leqq_R [g]$ in our notion.  Thus, using Borsuk's notion, \emph{$[X,Y]_R$} and \emph{$[X, Y]_L$} are the poset of the homotopy classes of \emph{co-conjugate} maps and \emph{conjugate} maps, resp. In this sense, \emph{$[X, Y]_{LR}$} is the poset of homotopy classes of \emph{conjugate-co-conjugate} maps, abusing words.
According to \cite{Hil2, Hil3}, R. Thom \cite{Thom} independently introduced the notion of \emph{dependence of cohomology classes}, but it turned out that Thom's dependence is subsumed in Borsuk's dependence, and the above results about the co-dependence marked the birth of \emph{Eckmann--Hilton duality}. 

\end{rem}

%%%%%%%%%%%%%
We can define orders on $[X,Y]_R, [X,Y]_L, [X,Y]_{LR}$. For the sake of completeness we write them down below.

\begin{pro}
The following orders are well-defined and they are partial orders, i.e., reflexive, antisymmetric and transitive.

\begin{enumerate}

\item For $[f]_R, [g]_R \in [X,Y]_{ R}$, $[f]_R \leqq'_R [g]_R \Leftrightarrow \exists \, \, [\phi] \in [X,X]\,\text{such that} \, [f] = [g] \circ [\phi]$, 
i.e. the following diagram commutes up to homotopy (namely, $f \sim g \circ \phi$):
$$\xymatrix
{
X \ar [d]_{\phi} \ar [r]^{f} & Y\\
X \ar [ur] _{g} & }
$$
\item For $[f]_L, [g]_L \in [X,Y]_{ L}$, $[f]_L \leqq'_L [g]_L  \Leftrightarrow \exists \, [\psi]\in [Y, Y]\, \text{such that} \, [f] = [\psi] \circ [g]$, 
i.e. the following diagram commutes up to homotopy (namely, $f \sim \psi \circ g$):
$$\xymatrix
{
X \ar[r]^{f} \ar[dr]_{g} & Y \\
 & Y \ar[u]_{\psi}
}
$$
\item For $[f]_{LR}, [g]_{LR} \in [X,Y]_{ LR}$, $[f]_{LR} \leqq'_{LR}  [g]_{LR} \Leftrightarrow \exists \, [\phi]\in [X, X], \exists [\psi] \in [Y, Y] \, \text{such that} \, [f] = [\psi] \circ [g] \circ [\phi]$, 
i.e. the following diagram commutes up to homotopy (namely, $f \sim \psi \circ g \circ \phi$):
$$\xymatrix
{
X \ar[d]_{\phi} \ar[r]^f & Y\\
X \ar [r]_g & Y \ar[u]_{\psi}.\\
}
$$
\end{enumerate}
\end{pro}

\begin{pro} The following canonical maps are monotone maps:
\begin{enumerate}
\item $\pi_R: ([X,Y], \leqq_R) \to ([X,Y]_R, \leqq'_R), \,\, \pi_R([f]):=[f]_R$
\item $\pi_L: ([X,Y], \leqq_L)  \to ([X,Y]_L, \leqq'_L), \, \,  \pi_L([f]):=[f]_L$
\item $\pi_{LR}:  ([X,Y], \leqq_{LR})  \to ([X,Y]_{LR}, \leqq'_{LR}),,  \, \, \pi_{LR}([f]):=[f]_{LR}$
\end{enumerate}
Hence each is a continuous map from a topological space (which is an Alexnadroff space) to a poset with the Alexandroff topology. In other words the homotopy set $[X,Y]$ can have these three poset-stratified space structures.
\end{pro}

\begin{thm} Let $h\mathcal Top$ be the homotopy category. 
\begin{enumerate}
\item For any object $S \in Obj(h\mathcal Top)$, we have an associated covariant functor
$\frak {st}^S_*: h\mathcal Top \to \mathcal Strat$
such that
\begin{enumerate}
\item for each object $Y \in Obj(h\mathcal Top)$, 
$$\frak {st}^S_*(X) := \left ( ([S,X], \tau_{\leqq_R}), ([S,X], \tau_{\leqq_R}) \xrightarrow {\pi_R} ([S,X]_R, \leqq'_R) \right)$$
\item for a morphism $[f] \in [X,Y]$ , $\frak {st}^S_*([f])$ is the following commutative diagram:
$$\xymatrix
{
([S,X], \tau_{\leqq_R}) \ar [r]^{\pi_R} \ar [d]_{f_*} & ([S,X]_R, \leqq'_R)  \ar [d]^{f_*} \\
([S,Y], \tau_{\leqq_R})  \ar[r]_{\pi_R} & ([S,Y]_R, \leqq'_R) \\
}
$$
\end{enumerate}
\item For any object $T \in Obj(h\mathcal Top)$, we have an associated contravariant functor
$\frak {st}_T^*: h\mathcal Top \to \mathcal Strat$
such that
\begin{enumerate}
\item for each object $X \in Obj(h\mathcal Top)$, 
$$\frak {st}_T^*(X) := \left ( ([X,T], \tau_{\leqq_L}), ([X,T], \tau_{\leqq_L}) \xrightarrow {\pi_L} ([X,T]_L, \leqq'_L) \right)$$
\item for a morphism $[f] \in [X,Y]$  , $\frak {st}_T^*([f])$ is the following commutative diagram:
$$\xymatrix
{
([Y,T], \tau_{\leqq_L}) \ar [r]^{\pi_L} \ar [d]_{f^*} & ([Y,T]_L, \leqq'_L)  \ar [d]^{f^*} \\
([X,T], \tau_{\leqq_L})  \ar[r]_{\pi_L} & ([X,T]_L, \leqq'_L) \\
}
$$
\end{enumerate}
\end{enumerate}
\end{thm}

\begin{ex}\label{ex1}
Let $X=Y_1=K(\bQ , 3)\times K(\bQ ,2)$ and $Y_2=K(\bQ , 3)\times K(\bQ ,5)$.
Recall the Sullivan minimal model $M(S)$ of a space $S$ \cite{FHT}.
Then  homotopy sets are identified with DGA(differential graded algebra)-homotopy sets as
$$(1)\ \ \ [X,Y_1]=[M(Y_1),M(X)]=[(\Lambda (x,y),0), (\Lambda (x,y),0)   ]$$
$$(2)\ \ \ [X,Y_2]=[M(Y_2),M(X)]=[(\Lambda (x,z),0), (\Lambda (x,y),0)   ]$$
where $|x|=3$, $|y|=2$ and $|z|=5$.
They are isomorphic to $\bQ\times \bQ=\{ (a,b)\mid a,b\in\bQ\}$ by the DGA-maps 
$f(x)=ax$ and $f(y)=by$  for (1)
and $f(x)=ax$ and $f(z)=bxy$ for (2), respectively.
 %in four points 
%\textcolor{red}
Then their right equivalence classes are 
(1) $ [X,Y_1]_{  R}=\{ \alpha , \beta, \gamma, \delta\}$ and 
(2) $ [X,Y_2]_{  R}=\{ \alpha' , \beta', \gamma', \delta'\}$ %, which are both four points.
 where $\alpha=\alpha'=[(0,0)]_R$,  $\beta=\beta'=[(1,0)]_R$,  $\gamma=\gamma'=[(0,1)]_R$ and  $\delta=\delta'=[(1,1)]_R$. 
However their
poset structures are given as the 
following Hasse diagrams:
$${\tiny 
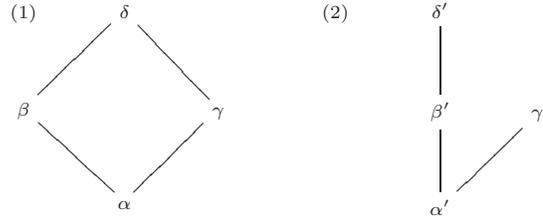
\xymatrix{
 {\large (1)}&\delta\ar@{-}[ld] \ar@{-}[dr]&&\\
\beta\ar@{-}[dr]&&\gamma \ar@{-}[dl] &\\
&\alpha%
} \xymatrix{
{\large  (2)}&\delta'&\\
&\beta'\ar@{-}[u]&\gamma'\\
&\alpha' \ar@{-}[u]\ar@{-}[ur]&\\
}}$$
respectively. 
In particular, there does not exist $\gamma'\leqq_R'\delta'$ in (2)  since
$\psi (M(f)(z))=\psi (xy)=0$ if $\psi (M(f)(x))=\psi (x)=0$
for $\psi :M(X)\to M(X)$.
For both cases, the stratifications of  $\bQ\times \bQ$ are given as  $$\bQ\times \bQ=e_{\alpha}\cup e_{\beta}\cup e_{\gamma}\cup e_{\delta}=e_{\alpha'}\cup e_{\beta'}\cup e_{\gamma'}\cup e_{\delta'}$$
where $e_{\alpha}=e_{\alpha'}=\{(0,0)\}$, $e_{\beta}=e_{\beta'}=\{ (a,0) \mid a\neq 0\}$,   $e_{\gamma}=e_{\gamma'}=\{ (0,b) \mid b\neq 0\}$
and  $e_{\delta}=e_{\delta'}=\{ (a,b) \mid ab\neq 0\}$.
However  the topologies are different.
Indeed, $\overline{e_{\delta}}=\bQ\times \bQ$ in (1) but
$\overline{e_{\delta'}}$ does not contain $e_{\gamma'}$ in (2).

If a map $f:Y_1\to Y_2$ is given by $M(f):(\Lambda (x,z),0)\to (\Lambda (x,y),0)$
with $M(f)(x)=x$ and $M(f)(z)=xy$, the induced map of homotopy sets 
$f_*:[X,Y_1]=\bQ\times \bQ\to [X,Y_2]=\bQ\times \bQ$
 is given by $f_*(a,b)=(a,ab)$.
Then the poset map 
$f_*:[X,Y_1]_R\to [X,Y_2]_R$ 
is given by $f_*(\alpha )=f_*(\gamma)=\alpha'$, $f_*(\beta)=\beta'$ and $f_*(\delta)=\delta'$.
%where the top point is the trivial fibrations. 
%However    $[X_i,Y_i]\cong \bQ\times \bQ$ in both cases.

\end{ex}

%%%%%%%%%%%%%%%%%%%%%%%%%%%%%%%%%%%%%%%%%%%%%%%%%%%%%%

\section{Some applications}
%%%%%%%%%
\begin{defn} For a group $G$ let $\mathcal Sub(G)$ be the set of all the subgroups of the group $G$. For subgroups $A, B \in \mathcal Sub(G)$ we define the order $A \leqq B $ by $A \subseteq B$, which is a partial order. 
\end{defn}
%%%%%%%%%%%%%
\begin{lem} Let $H_*(-)$ be the homology theory with a coefficient ring $\mathscr R$. Then the following maps are well-defined and monotone (order-preserving) maps:
\begin{enumerate}
\item $\op{Im}_{H_*}:([X,Y], \leqq_R) \to \left (\mathcal Sub(H_*(Y)), \leqq \right), \, \op{Im}_{H_*}([f]):= \op{Im} (f_*: H_*(X) \to H_*(Y)).$
\item $\op{Im}'_{H_*}:([X,Y]_R, \leqq'_R) \to \left (\mathcal Sub(H_*(Y)), \leqq \right), \, \op{Im}'_{H_*}([f]_R):= \op{Im}_{H_*}([f]).$
\end{enumerate}
We have the following commutative diagram:
$$\xymatrix
{
([X,Y], \leqq_R) \ar [d]_{\op{id}_{[X,Y]}} \ar [r]^{\pi_R} & ([X,Y]_R , \leqq'_R) \ar[d]^{\op{Im}'_{H_*}} \\
([X,Y], \leqq_R)  \ar[r]_{\op{Im}_{H_*}\, \, \, } & \left (\mathcal Sub(H_*(Y)), \leqq \right).
}
$$

\end{lem}
%%%%%%%%%%
\begin{proof} Let $[f] \leqq_R [g]$. Thus $\exists t:X \to X$ such that $f \sim g \circ t.$ Hence $f_* = g_* \circ t_*$, i.e., the following diagram commutes:
$$\xymatrix
{
 H_*(X) \ar [d]_{t_*} \ar [r]^{f_*} & H_*(Y),\\
 H_*(X) \ar[ur]_{g_*} & }
$$

\noindent
which implies that
$\op{Im} (f_*:  H_*(X) \to H_*(Y)) \subset \op{Im} (g_*: H_*(X) \to H_*(Y))$. Thus $\op{Im}_{H_*}([f]) \subset \op{Im}_{H_*}([g]).$ Hence $\op{Im}_{H_*}:([X,Y], \leqq_R) \to \left (\mathcal Sub(H_*(Y)), \leqq \right)$ is a monotone map.
For $\op{Im}'_{H_*}$ we just observe that if $[f] \sim_R[g]$, i.e., $\exists t_1:X \to X, t_2:X \to X$ sucht that $f \sim g \circ t_1$ and $g \sim g \circ t_2$, then it follows from the above that $\op{Im} (f_*: H_*(X) \to H_*(Y)) = \op{Im} (g_*: H_*(X) \to H_*(Y))$, i.e., $\op{Im}_{H_*}([f]) = \op{Im}_{H_*}([g]).$ Thus $\op{Im}_{H_*}([f]_R):= \op{Im}_{H_*}([f])$ is well-defined.
\end{proof}

%%%%%%%%%%%%%
Similarly we get the following:
\begin{lem} Let $H^*(-)$ be the cohomology theory with a coefficient ring $\mathscr R$. Then the following maps are well-defined and monotone (order-preserving) maps:
\begin{enumerate}
\item $\op{Im}_{H^*}:([X,Y], \leqq_L) \to \left (\mathcal Sub(H^*(X)), \leqq \right), \, \op{Im}_{H^*}([f]):= \op{Im} (f^*: H^*(Y) \to H_*(X)).$
\item $\op{Im}'_{H^*}:([X,Y]_L, \leqq'_L) \to \left (\mathcal Sub(H^*(X)), \leqq \right), \, \op{Im}'_{H^*}([f]_L):= \op{Im}_{H^*}([f]).$
\end{enumerate}
We have the following commutative diagram:
$$
\xymatrix
{
([X,Y], \leqq_L) \ar [d]_{\op{id}_{[X,Y]}} \ar [r]^{\pi_L} & ([X,Y]_L , \leqq'_L) \ar[d]^{\op{Im}'_{H^*}} \\
([X,Y], \leqq_L)  \ar[r]_{\op{Im}_{H^*}\, \, \, } & \left (\mathcal Sub(H^*(X)), \leqq \right).
}
$$
\end{lem}
%%%%%%%%%%

%%%%%%%%%%%%
\begin{cor}  Let $H_*(-)$ and $H^*(-)$ be as above.
\begin{enumerate}
\item
For $\forall S \in Obj(h\mathcal Top)$, we have a {\bf \emph{covariant}} functor
$\frak {st}^S_{ H_*}: h\mathcal Top \to \mathcal Strat$
such that
\begin{enumerate}
\item for each object $X \in Obj(h\mathcal Top)$, 

$\frak {st}^S_{H_*}(X) :=  \left ( ([S,X], \tau_{\leqq_R}), ([S,X], \tau_{\leqq_R}) \xrightarrow {\op{Im}_{H_*}} \left (\mathcal Sub(H_*(X)), \leqq \right) \right)$.
\item for a morphism $[f] \in [X,Y]$ , $\frak {st}^S_{H_*}([f])$ is the following commutative diagram:
$$
\xymatrix
{
([S,X], \tau_{\leqq_R}) \ar [r]^{\op{Im}_{H_*} \quad } \ar [d]_{f_*} & \left (\mathcal Sub(H_*(X)), \leqq \right) \ar [d]^{f_*} \\
([S,Y], \tau_{\leqq_R})  \ar[r]_{\op{Im}_{H_*} \quad } & \left (\mathcal Sub(H_*(Y)), \leqq \right).
}
$$
\end{enumerate}
\item $\op{Im}'_{H_*}$ gives rise to a natural transformation $\op{Im}'_{H_*}: \frak {st}^S_*(-) \to \frak {st}^S_{H_*}(-)$, namely for a morphism $[f] \in [X, Y]$ we have the following commutative diagram:
$$
\xymatrix
{
\frak {st}^S_*(X) \ar [d]_{f_*} \ar [r]^{\op{Im}'_{H_*}} & \frak {st}^S_{H_*}(X) \ar[d]^{f_*} \\
\frak {st}^S_*(Y)  \ar[r]_{\op{Im}'_{H_*}} & \frak {st}^S_{ H_*}(Y).}
$$
Namely we have the following commutative cube:
$$ 
\xymatrix
{ ([S,X], \tau_{\leqq_R}) \ar[dd]_{f_*} \ar[rd]^{\pi_R} \ar[rr]^{\op{id}_{[S,X]}} && ([S,X], \tau_{\leqq_R}) \ar'[d][dd]^{f_*} \ar[rd]^{\op{Im}_{H_*}} \\
& ([S,X]_R, \leqq'_R) \ar[dd]_{f_*} \ar[rr] ^{\op{Im}'_{H_*} \quad \quad \quad \quad  }  &&  \left (\mathcal Sub(H_*(X)), \leqq \right) \ar[dd]^{f_*} \\
([S,Y], \tau_{\leqq_R}) \ar'[r] [rr]_{\op{id}_{[S,Y]} \quad \quad }  \ar[rd]_{\pi_R} && ([S,Y], \tau_{\leqq_R})  \ar[rd]_{\op{Im}_{H_*}} \\
& ([S,Y]_R, \leqq'_R) \ar[rr] _{\op{Im}'_{H_*} \quad \quad  \quad }  &&  \left (\mathcal Sub(H_*(Y)), \leqq \right).
}
$$

\item For any object $T \in Obj(h\mathcal Top)$, we have an associated {\bf \emph{contravariant}} functor
$\frak {st}_T^{H^*}: h\mathcal Top \to \mathcal Strat$
such that
\begin{enumerate}
\item for each object $X \in Obj(h\mathcal Top)$, 
$$\frak {st}_T^{H^*}(X) := \left ( ([X,T], \tau_{\leqq_L}), ([X,T], \tau_{\leqq_L}) \xrightarrow {\op{Im}_{H^*}} \left (\mathcal Sub(H^*(X)), \leqq \right)  \right)$$
\item for a morphism $[f] \in [X,Y]$  , $\frak {st}_T^{H^*}([f])$ is the following commutative diagram:
$$
\xymatrix
{
([Y,T], \tau_{\leqq_L}) \ar [r]^{\op{Im}_{H^*} \quad } \ar [d]_{f^*} & \left (\mathcal Sub(H^*(Y)), \leqq \right)  \ar [d]^{f^*} \\
([X,T], \tau_{\leqq_L})  \ar[r]_{\op{Im}_{H^*} \quad } & \left (\mathcal Sub(H^*(X)), \leqq \right) .
}
$$
\end{enumerate}
\item $\op{Im}'_{H^*}$ gives rise to a natural transformation $\op{Im}'_{H^*}: \frak {st}^*_T(-) \to \frak {st}_T^{H^*}(-)$, namely for a morphism $[f] \in [X, Y]$ we have the following commutative diagram:
$$\xymatrix
{
\frak {st}_T^*(Y) \ar [d]_{f^*} \ar [r]^{\op{Im}'_{H^*}} & \frak {st}_T^{H^*}(Y)\ar[d]^{f_*} \\
\frak {st}_T^*(X)  \ar[r]_{\op{Im}'_{H^*}} & \frak {st}_T^{H^*}(X).}
$$
Namely we have the following commutative cube:
$$
\xymatrix
{ ([Y,T], \tau_{\leqq_L}) \ar[dd]_{f^*} \ar[rd]^{\pi_L} \ar[rr]^{\op{id}_{[Y,T]}} && ([Y,T], \tau_{\leqq_L}) \ar'[d][dd]^{f^*} \ar[rd]^{\op{Im}_{H^*}} \\
& ([Y,T]_L, \leqq'_L) \ar[dd]_{f^*} \ar[rr] ^{\op{Im}'_{H^*}\quad \quad \quad \quad  }  &&  \left (\mathcal Sub(H^*(Y)), \leqq \right)  \ar[dd]^{f^*} \\
([X,T], \tau_{\leqq_L}) \ar'[r] [rr]_{\op{id}_{[X,T]} \quad \quad }  \ar[rd]_{\pi_L} && ([X,T], \tau_{\leqq_L})  \ar[rd]_{\op{Im}_{H^*}} \\
& ([X,T]_L, \leqq'_L) \ar[rr] _{\op{Im}'_{H^*} \quad \quad  \quad }  &&  \left (\mathcal Sub(H^*(X)), \leqq \right).
}
$$

\end{enumerate}
\end{cor}
%%%%%%%%%%%%% 
%%%%%%%%%%%%%%%% A connection to Thom's dependence %%%%%%
The case of $\op{Im}_{H^*}: ([X,T], \leqq_L) \to \left(\mathcal Sub (H^*(X), \leqq \right )$ is related to Thom's dependence of cohomology classes \cite{Thom} mentioned %above.
in the introduction. 
To explain this, we recall the definition of dependence of cohomology classes (e.g., see \cite{Hil1}).

\begin{defn}[R. Thom] The cohomology class $\beta \in H^q(X; B)$ \emph{depends on} the cohomology class $\alpha \in H^p(X; A)$, where $A, B$ are coefficient rings, if, for all (perhaps infinite) polyhedra $Y$ and all maps $f: X \to Y$ such that $\alpha \in f^*(H^p(Y; A))$, we have $\beta \in f^*(H^q(Y;B))$.
\end{defn}

Thom \cite{Thom} proves the following proposition (see \cite{Hil1}). For this we recall that the cohomology theory is representable by the Eilenberg-Maclane space, i.e., $H^j(X, \Lambda) \cong [X, K(\Lambda,j)]$ where $K(R,j)$ is the Eilenberg-Maclane space whose homotopy type is completely characterized by the homotopy groups $\pi_j(K(\Lambda,j)) = \Lambda$ and $\pi_i(K(\Lambda,j))= 0, i \not = j.$ Then by the Hurewicz Theorem we have $H_j(K(\Lambda,j); \mathbb Z)\cong \pi_j(K(\Lambda,j)) = \Lambda$ and $H_d(K(\Lambda,j)) = 0$ for $d<j$. Hence by the universal coefficinet theorem we have the isomorphism
$$\Phi:H^j(K(\Lambda,j);\Lambda) \cong Hom(H_j(K(\Lambda,j);\mathbb Z), \Lambda) \cong Hom (\pi_j(K(\Lambda,j)) , \Lambda) \cong Hom(\Lambda, \Lambda).$$
Let $u :=\Phi^{-1}(\op{id}_{\Lambda})$ for the identity map $\op{id}_{\Lambda}: \Lambda \to \Lambda$. Then the isomorphism $\Theta: [X, K(\Lambda,j)] \cong H^j(X, \Lambda)$ is obtained by $\Theta ([f]):=f^*u$ where $f^*:H^j(K(\Lambda,j);\Lambda) \to H^j(X,\Lambda)$.
%%%%%%%%%%%%%%
\begin{pro}[R. Thom \cite{Thom}] Let $\alpha \in H^p(X; A) \cong [X, K(A, p)]$ and let $f_{\alpha}:X \to K(A, p)$ be a map such that the homotopy class $[f_{\alpha}]$ corresponds to $\alpha$.
Then $\beta \in H^q(X,B)$ depends on $\alpha$ if and only if $\beta \in f^*_{\alpha}(H^q(K(A,p); B))$.
\end{pro}

Using this proposition we can get the following result. By the monotone (order-preserving) map
$$\op{Im}_{H^*}:([X, K(A,p)], \leqq_L) \to \left (\mathcal Sub(H^*(X;B)), \leqq \right)$$
the image $\op{Im}_{H^*}([f_{\alpha}]) =f_{\alpha}^*((H^q(K(A,p); B))$ is \emph{nothing but the subgroup of all the cohomology classes $\beta \in H^q(X;B)$ depending on the cohomology class $\alpha$}.

We also see that let $\alpha, \alpha' \in H^p(X, A)$ and let $f_{\alpha}, f_{\alpha'} :X \to K(A, p)$ be the corresponding maps. Then, if $f_{\alpha}$ depends on $f_{\alpha'}$, i.e., $[f_{\alpha}] \leqq_L [f_{\alpha'}]$ by our terminology (in other words, we can define the order of the cohomology classes $\alp \leqq_L \alp'$ by this),  then we have $(\alpha \in) \op{Im}_{H^*}([f_{\alpha}]) \subset \op{Im}_{H^*}([f_{\alpha'}])$, i.e., $\op{Im}_{H^*}([f_{\alpha}]) \leqq \op{Im}_{H^*}([f_{\alpha'}])$.  Thus, that $\alpha$ depends on $\alpha'$ is equivalent to that $\op{Im}_{H^*}([f_{\alpha}]) \leqq \op{Im}_{H^*}([f_{\alpha'}])$.\\
%%%%%%%%%%%

Here is another application to vector bundles and characteristic classes (e.g., see \cite{MS}, \cite{Ha}). 
Let $\op{Vect}_n(X)$ be the set of isomorphism classes of complex vector bundles of rank $n$. Then it is well-known that
$$\op{Vect}_n(X) \cong [X, G_n(\mathbb C^{\infty})]$$
where $G_n(\mathbb C^{\infty})$ is the infinite Grassmann manifold of complex planes of dimension $n$, i.e., the classifying space of complex vector bundles of rank $n$.
This isomorphism is by the correspondence $[E] \longleftrightarrow [f_E]$, where $f_E:X \to G_n(\mathbb C^{\infty})$ is a classifying map of $E$, i.e., $E = f_E^* \gamma^n,$ 
where $\gamma^n$ is the universal complex vector bundle of rank $n$ over $G_n(\mathbb C^{\infty})$.

By the isomorphism  $\op{Vect}_n(X) \cong [X, G_n(\mathbb C^{\infty})]$ we can consider the peorder of $[E]$ and $[F]$:
$$[E] \leqq_L [F] \Longleftrightarrow [f_E] \leqq_L [f_F],$$
where $f_E, f_F:X \to G_n(\mathbb C^{\infty})$ are respectively the classifying maps of $E$ and $F$.

Then we have the following well-defined monotone (order-preserving) map:
$$ \op{Im}_{H^*}: (\op{Vect}_n(X), \leqq_L) \to (\mathcal Sub(H^*(X; \mathbb Z)), \leqq )$$
defined by 
$ \op{Im}_{H^*}([E]):= \op{Im}\Bigl (f_E^*:H^*(G_n(\mathbb C^{\infty});\mathbb Z) \to H^*(X; \mathbb Z) \Bigr ).$
By the definition of characteristic classes, for each element $\alp \in H^*(G_n(\mathbb C^{\infty}))$, the pullback $f_E^*(\alp)$ is called the characteristic class of $E$ defined by the class $\alp$, and denoted by
$\alp(E) := f_E^*(\alp)$. It is well-known (e.g., see \cite{MS}) that $H^*(G_n(\mathbb C^{\infty})) = \mathbb Z[c_1, c_2, \cdots, c_n]$ is generated by $1$ and the Chern classes $c_1, c_2, \cdots, c_n$ of the universal bundle $\ga^n$. Here $1, c_1, c_2, \cdots, c_n$ are linearly independent. \emph{ $\op{Im}\Bigl (f_E^*:H^*(G_n(\mathbb C^{\infty});\mathbb Z) \to H^*(X; \mathbb Z) \Bigr )$ is nothing but the subring consisting of \underline{all} the characteristic classes of $E$}, which %is 
could be also denoted by $\mathbb Z [c_1(E), c_2(E), \cdots, c_n(E)]$. Here we should note that $1, c_1(E), c_2(E), \cdots, c_n(E)$ are \emph{not} linearly independent in general. Let us denote this subring by $\op{Char}(E)$.
Therefore we have
$[E] \leqq_L[F] \Longrightarrow \op{Char}(E) \subseteq \op{Char}(F).$
We also get that 
$[E] \sim_L[F] \Longrightarrow \op{Char}(E) = \op{Char}(F).$

\begin{rem} In the case of real vector bundles, the complex infinite Grassmann $G_n(\mathbb C^{\infty})$, the Chern class $c_i$ and the coefficient ring $\mathbb Z$ are respectively replaced by the real infinite Grassmann 
 $G_n(\mathbb R^{\infty})$, the Stiefell--Whitney class $w_i$ and the coefficient ring $\mathbb Z_2$.
\end{rem}

\begin{rem} Instead of homology $H_*(-)$ and cohomology $H^*(-)$, we can consider homotopy version of these, i.e., homotopy groups $\pi_*(-)$ and cohomotopy ``groups" $\pi^*(-)$. In this case we consider the based homotopy set $[X,Y]_*$. We note that the cohomotopy set $\pi^p(X) :=[X, S^p]$ (e.g., see \cite{SHu}). Note that in the case when $p=1$, $\pi^1(X) =[X, S^1] =[X, K(\mathbb Z,1)] =H^1(X;\mathbb Z)$ is an abelian group.
\end{rem}

\begin{rem} For any locally small category $\mathcal C$, in a similar manner as above we can consider a poset-stratified space structure on the hom set $hom_{\mathcal C}(X,Y)$ for any objects $X, Y \in Obj(\mathcal C)$, and using reasonable covariant functor $\mathcal H_*$ and contravariant functor $\mathcal H^*$ on the locally small category $\mathcal C$ we can do similar things as above. For example, derived categories, triangulated categories, and derived functors, etc.
\end{rem}

%%%%%% Gottlieb group
When it comes to the homotopy groups $\pi_*$, we have another application. Let $Map(X,Y;f)$ be the path component of $Map(X,Y)$ containing $f$. Let $*$ be the base point of $X$ and we consider the evaluation map
$$ev:Map(X, Y;f) \to Y \quad ev(g):= g(*).$$
\begin{defn}[\cite{WL}] For a continuous based map $f:X \to Y$, \emph{the $n$-th evaluation subgroup $G_n(Y, X;f)$} of the $n$-th homotopy group $\pi_n(Y)$ is defined as follows:
$$G_n(Y, X;f):= \op{Im} \Bigl ( ev_*: \pi_n(Map(X,Y;f)) \to \pi_n(Y) \Bigr ).$$
\end{defn}
This is a generalized version of the following Gottlieb group $G_n(X)$ (\cite{Gott1, Gott2}):
$$G_n(X):= \op{Im} \Bigl ( ev_*: \pi_n(aut_1X) \to \pi_n(X) \Bigr ),$$
where $aut_1X = Map(X, X; \op{id}_X)$ and $id_X$ is the identity map. 

The $n$-th evaluation subgroup $G_n(Y, X;f)$ can be described as follows:

\begin{lem}[\cite{WL}]
The $n$-th {\it evaluation subgroup} of a continuous based map $f:X\to Y$ is  
$$G_n(Y,X;f):=\Biggl \{ a\in \pi_n(Y) \mid 
\xymatrix{
X\times S^n\ar@{.>}[dr]^{ \exists\phi} & S^n\ar[l]_{\ \ i_{S^n}}\ar[d]^a\\
X \ar [r]_f\ar[u]^{i_X} & Y 
}\mbox{ is homotopy commutative}
\Biggr\}$$
from the adjointness. 
\end{lem}

As to the case of generalized Gottlieb groups, we need to reverse the order.

\begin{pro}
 The following map (called ``the $n$-th generalized Gottlieb evaluation subgroup map")
$$\frak g_n: [X, Y] \to \mathcal S(\pi_n(Y)) \quad G_n([f]):= G_n(Y,X;f)$$
is well-defined, i.e., $f \sim f'$ implies that $G_n(Y,X;f) = G_n(Y,X;f')$.
\end{pro}

\begin{pro}\label{g.gott.map}
The following map (called ``the finer $n$-th generalized Gottlieb evaluation subgroup map")
$$\frak g_n^R: [X, Y]_R \to \mathcal S(\pi_n(Y)) \quad \frak g_n^R([f]_R):= G_n(Y,X;[f])=G_n(Y,X;f)$$
is well-defined, i.e., $[f] \sim_{R} [g]$ implies that $G_n(Y,X;f)=G_n(Y,X;g)$.
Namely the following diagram commutes:
$$\xymatrix
{
[X,Y] \ar [dr]_{\frak g_n} \ar [r]^{\pi_R} & [X,Y]_R \ar [d]^{\frak g_n^R}\\
& \mathcal S(\pi_n(Y)).  & }
$$
\end{pro}
\begin{proof}
For two maps $f,\ g:X\to Y$, suppose that $f\sim g\circ h$ for some map $s:X\to X$.
Then  $G_n(Y,X;g)\subset G_n(Y,X;f)$.
Indeed, there is the homotopy commutative diagram for $a\in G_n(Y,X;g)$:
$$\xymatrix{
X\ar[dr]^f\ar[rr]^{i_X} \ar[dd]_s && X\times S^n\ar[dd]^(.3){s\times 1} | (.5)\hole\ar@{.>}[ld]_{\psi}&\\
&Y&&S^n\ar[ul]_{i_{S^n}}\ar[dl]^{i_{S^n}}\ar[ll]^(.3)a\\
X \ar [rr]_{i_X}\ar[ru]_{g} &&  X\times S^n\ar[ul]_{ \exists\phi} &
}
$$
by $\psi:=\phi\circ (s\times 1)$.
Then $\psi\circ i_X\simeq f$ and $\psi\circ i_{S^n}\simeq a$.
Hence $a\in G_n(Y,X;f)$.
Furthermore,  suppose that $g\sim f\circ s'$.
Then similarly we obtain $G_n(Y,X;f)\subset G_n(Y,X;g)$.
\end{proof}

As a corollary of the above proof, we have the following
\begin{thm}
\begin{enumerate}
\item If $[f] \leqq_R [g]$, then we have $G_n(Y,X;g)\subset G_n(Y,X;f)$, i.e., $\frak g_n([g]) \leqq \frak g_n([f]).$ Hence

$\frak g_n: ([X,Y], \leqq^{op} _R) \to \mathcal Sub(\pi_n(Y), \leqq)$ is a monotone map.
\item 
If $[f]_R\leqq[g]_R$, then we have $G_n(Y,X;g)\subset G_n(Y,X;f)$, i.e., $\frak g_n^R([g]_R) \leqq \frak g_n^R([f]_R).$ Hence

$\frak g_n^R: ([X,Y]_R, \leqq'^{op} _R) \to \mathcal Sub(\pi_n(Y), \leqq)$ is a monotone map.

\end{enumerate}

We also have the following commutative diagram:
\begin{equation}\label{cd}
\xymatrix
{
([X,Y], \leqq^{op}_R) \ar[d]_{\op{id}_{[X,Y]}} \ar [r]^{\pi_R}  & ([X,Y]_R, \leqq'^{op}_R) \ar[d]_{\frak g_n^R}\\
([X,Y], \leqq^{op}_R) \ar[r]_{\frak g_n} & (\mathcal Sub(\pi_n(Y)), \leqq) }.
\end{equation} 
\end{thm}

%%%%%%%%%%%%%%%%%%
\begin{cor}  
\begin{enumerate}
\item
For $\forall S \in Obj(h\mathcal Top)$, we have a {\bf \emph{covariant}} functor
$\frak {st}^S_{Gott}: h\mathcal Top \to \mathcal Strat$
such that
\begin{enumerate}
\item for each object $X \in Obj(h\mathcal Top)$, 

$\frak {st}^S_{Gott}(X) :=  \left ( ([S,X], \tau_{\leqq^{op}_R}), ([S,X], \tau_{\leqq^{op}_R}) \xrightarrow {\frak g_n} \left (\mathcal Sub(\pi_n(X)), \leqq \right) \right)$.
\item for a morphism $[f] \in [X,Y]$ , $\frak {st}^S_{Gott}([f])$ is the following commutative diagram:
$$
\xymatrix
{
([S,X], \tau_{\leqq^{op}_R}) \ar [r]^{\frak g_n \quad } \ar [d]_{f_*} & \left (\mathcal Sub(\pi_n(X)), \leqq \right) \ar [d]^{f_*} \\
([S,Y], \tau_{\leqq^{op}_R})  \ar[r]_{\frak g_n \quad } & \left (\mathcal Sub(\pi_n(Y)), \leqq \right).
}
$$
\end{enumerate}
\item $\frak g_n^R$ gives rise to a natural transformation $\frak g_n^R: \frak {st}^S_*(-) \to \frak {st}^S_{Gott}(-)$, namely for a morphism $[f] \in [X, Y]$ we have the following commutative diagram:
$$
\xymatrix
{
\frak {st}^S_*(X) \ar [d]_{f_*} \ar [r]^{\frak g_n^R} & \frak {st}^S_{Gott}(X)\ar[d]^{f_*} \\
\frak {st}^S_*(Y)  \ar[r]_{\frak g_n^R} & \frak {st}^S_{Gott}(Y)}.
$$
Namely we have the following commutative cube:
$$ \small
\xymatrix
{ ([S,X], \tau_{\leqq^{op}_R}) \ar[dd]_{f_*} \ar[rd]^{\pi_R} \ar[rr]^{\op{id}_{[S,X]}} && ([S,X], \tau_{\leqq^{op}_R}) \ar'[d][dd]^{f_*} \ar[rd]^{\frak g_n} \\
& ([S,X]_R, \leqq'^{op}_R) \ar[dd]_{f_*} \ar[rr] ^{\frak g_n^R \quad \quad \quad \quad  }  &&  \left (\mathcal Sub(\pi_n(X)), \leqq \right) \ar[dd]^{f_*} \\
([S,Y], \tau_{\leqq^{op}_R}) \ar'[r] [rr]_{\op{id}_{[S,Y]} \quad \quad }  \ar[rd]_{\pi_R} && ([S,Y], \tau_{\leqq^{op}_R})  \ar[rd]_{\frak g_n} \\
& ([S,Y]_R, \leqq'^{op}_R) \ar[rr] _{\frak g_n^R \quad \quad  \quad }  &&  \left (\mathcal Sub(\pi_n(Y)), \leqq \right).
}
$$
\end{enumerate}
\end{cor}

%%%%%%%%%%%%%%

\begin{rem} When it comes to the case $[X,Y]_L$ we do not have similar results as above.
\end{rem}

Let $G_*(Y, X;f):= \bigoplus_n G_n(Y,X;f) \subset \pi_*(Y):=\bigoplus\pi_n(Y)$. We let
$$\mathcal G(X,Y):= \{ G_*(Y, X; f) \, | \, f \in Map(X, Y) \}$$ be the poset with the partial order by the inclusions $G_*(Y,X;g)\subset G_*(Y,X;f)$ for some maps $f$ and $g$ from $X$ to $Y$.
Then $\pi_*(Y)=G_*(Y,X;*)$ is the maximal element of  ${\mathcal G}(X,Y)$.
In particualr, when $X=Y$, the Gottlieb group $G_*(X) :=G_*(X,X;id_X)$ is the minimal element of ${\mathcal G}(X,X)$.
Thus
\begin{cor}
The map  $G:([X,Y]_R, \leqq'^{op}_R) \to ({\mathcal G}(X,Y), \leqq)$ given by $G([f]_R)=G(f):=G_*(Y,X;f)$ is a poset map.
\end{cor}

\begin{ex}
Let $X=S^n$ and $Y=(S^n\times S^n)_0$ for an even integer $n$.
Here $(S^n\times S^n )_0$ is the rationalization of $S^n\times S^n$ \cite{HMR}.
Then $[X,Y]_R=\bQ \oplus \bQ/\sim_R \ =P^1(\bQ )\cup (0,0)$ as a set 
with $(a,b)\sim_R(a',b')$ when $a'=ka$ and $b'=kb$ for some $k\in \bQ-0$.
It is  ordered only by
$[a,b]<(0,0)$ for any $[a,b]\in P^1(\bQ)$.  
On the other hand, ${\mathcal G}(X,Y)$ is the set of four points whose order is given as the Hasse diagram: 
$$
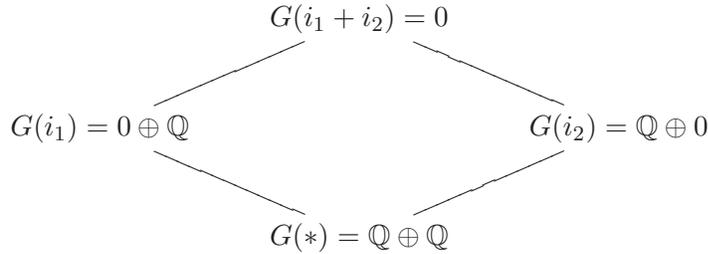
\xymatrix{
 &G(i_1+i_2)=0&\\
G(i_1)=0 \oplus \bQ\ar@{-}[ur]&& G(i_2)=\bQ \oplus 0\ar@{-}[ul]\\
&G(*)=\bQ \oplus \bQ \ar@{-}[ul]\ar@{-}[ur]&\\
} $$
for the $k$-factor inclusion $i_k:S^n\to (S^n\times S^n)_0$ and the constant map $*$.
Then the poset map $G:[X,Y]_R\to {\mathcal G}(X,Y)$ is given by
$G((0,0))=\bQ \oplus \bQ$, $G([1,0])=0 \oplus \bQ$, $G([0,1])=\bQ \oplus 0$ and $G([a,b])=0$ when $ab\neq 0$.
%There is a poset map $(\bZ \oplus \bZ,<)\to $
\end{ex}

%%%%%%%%%%%%%%%%%%%%%%%%%%%%%%%%coGottlieb
\begin{defn}\cite[Definition 2.1]{Yoon}
The $n$-th {\it generalized dual Gottlieb set} of a map $f:X\to Y$ is  
$$G^n(X,f,Y):=\Biggl \{ a\in H^n(X) \mid 
\xymatrix{
X\ar[d]_f\ar[rr]^{(f\times a)\circ \Delta\ \ \ }\ar@{.>}[drr]^{ \exists\phi} && \ \ Y\times K({\mathbb Z}, n)\\
Y \ar [rr]_{i_Y\ \ \ } &&\ \  Y \vee K({\mathbb Z} ,n)\ar[u]_{incl. }
}
\, \, \text{is homotopy commutative} \Biggr \}
$$
%{\ \ \ \ \ \ \ \ \ \ \ \ \ \ \ \ \ \ \ \ \ \ \ \ \ \ \ \ \  is homotopy commutative}$\}$
for the diagonal map $\Delta :X\to X\times X$.
\end{defn}

\begin{pro}\label{dg.gott.map}
The following map (called ``the finer $n$-th generalized dual Gottlieb  map")
$$\frak g^n_L: [X, Y]_L \to \mathcal S(H^n(X)) \quad \frak g^n_L([f]_L):= G^n(X,f,Y)$$
is well-defined, i.e., $[f] \sim_{L} [g]$ implies that $G^n(X,f,Y)=G^n(X,g,Y)$.
Namely the following diagram commutes:
$$\xymatrix
{
[X,Y] \ar [dr]_{\frak g^n} \ar [r]^{\pi_L} & [X,Y]_L \ar [d]^{\frak g^n_L}\\
& \mathcal S(H^n(X))  & }
$$
\end{pro}
\begin{proof}
For two maps $f,\ g:X\to Y$, suppose that $g\sim s\circ f$ for some map $s:Y\to Y$.
Then  $G^n(X,f,Y)\subset G^n(X,g,Y)$.
Indeed, there is the homotopy commutative diagram for $a\in G^n(X,f,Y)$:
$$\xymatrix{
Y\ar[rr]^{i_Y} \ar[dd]_s && Y\vee K(\bZ ,n)\ar[dd]^(.3){s\vee 1} | (.5)\hole&\\
&X\ar[ul]_f\ar[dl]^g\ar@{.>}[rd]_{ \phi}\ar[ru]^{\exists \psi}\ar[rr]_{\ \ \ a}&&K(\bZ,n)\ar[ul]_{i_{K}}\ar[dl]^{i_{K}}\\
Y \ar [rr]_{i_Y} &&  Y\vee K(\bZ,n)&
}
$$
by $\phi:= (s\vee 1)\circ \psi$.
Then $i_X\circ g\simeq \phi$ and $i_{S^n}\circ a\simeq \phi$.
Hence $a\in G^n(X,g,Y)$.
Furthermore,  suppose that $f\sim s'\circ g$.
Then similarly we obtain $G^n(X,g,Y)\subset G^n(X,f,Y)$.
\end{proof}

\begin{rem}
For  generalized dual Gottlieb  sets, we obtain similar properties as evaluation subgroups. 
\end{rem}

%%%%%%%%%%%%%%%%%%%%%%%%%%%%%% cat

\begin{ex}
Let ${\rm cat}(f)$ be the Lusternik--Schnirelmann category of a map $f:X\to Y$ (\cite[p.352]{FHT}). Then $\op{cat}:[X,Y] \to (\bZ_{\geqq 0},\leqq )$ is a monotone map.
In the case of cat, we have the three finer poset-stratified space structure on the reversed ordered posets $[X, Y]_R$, $[X,Y]_L$ and $[X,Y]_{LR}$ as follows:
\begin{enumerate}
\item If $[g] \leqq_R [f]$, i.e., $g\sim  f\circ s$ with $s:X \to X$, then we have (\cite[Lemma 27.1(ii)]{FHT})
$${\rm cat}(g)= {\rm cat}(f\circ s)\leqq \min \{ {\rm cat}(f), {\rm cat}(s)\}\leqq {\rm cat}(f).$$
Hence we have  ${\rm cat}(g)\leqq {\rm cat}(f)$.
Thus there is a poset map ${\rm cat}_R:[X,Y]_R \to (\bZ_{\geqq 0},\leqq ) $. Here $\op{cat}_R([f]_R):=\op{cat}(f).$ 

\item If $[g] \leqq_L [f]$, i.e., $g\sim  t \circ f$ with $t:Y \to Y$, then we have
$${\rm cat}(g)= {\rm cat}(t\circ f)\leq \min \{ {\rm cat}(t), {\rm cat}(f)\}\leqq {\rm cat}(f).$$
Hence we have  ${\rm cat}(g)\leqq {\rm cat}(f)$.
Thus ${\rm cat}_L:[X,Y]_L\to (\bZ_{\geqq 0},\leqq ) $ is a poset map. Here $\op{cat}_L([f]_R):=\op{cat}(f).$ 

\item If $[g] \leqq_{LR} [f]$, i.e., $g\sim h\circ f\circ s$ with $s:X \to X$ and $t:Y \to Y$, then we have
$${\rm cat}(g)= {\rm cat}(t\circ f\circ s)\leq \min \{ {\rm cat}(t), {\rm cat}(f), {\rm cat}(s)\}\leqq {\rm cat}(f).$$
Hence we have  ${\rm cat}(g)\leqq {\rm cat}(f)$.
Thus ${\rm cat}_{LR}:[X,Y]_{LR}\to (\bZ_{\geqq 0},\leqq ) $ is a poset map. Here $\op{cat}_{LR}([f]_R):=\op{cat}(f).$ 
\end{enumerate}
%\begin{enumerate}
%\item When $[g] \leqq_R [f]$, 
% from
%$${\rm cat}(g)= {\rm cat}(f\circ s)\leqq \min \{ {\rm cat}(f), {\rm cat}(s)\}\leqq {\rm cat}(f)$$
%for some map $s:X\to X$ with $g\sim  f\circ s$ (\cite[Lemma 27.1(ii)]{FHT}), we have  ${\rm cat}(g)\leqq {\rm cat}(f)$.
%Thus there is a poset map ${\rm cat}_R:[X,Y]_R \to (\bZ_{\geqq 0},\leqq ) $. Here $\op{cat}_R([f]_R):=\op{cat}(f).$ 

%\item When  $[g] \leqq_L [f]$, 
% from
%$${\rm cat}(g)= {\rm cat}(t\circ f)\leq \min \{ {\rm cat}(t), {\rm cat}(f)\}\leqq {\rm cat}(f)$$
%for some map $t:Y\to Y$ with $g\sim t\circ f$, we have  ${\rm cat}(g)\leqq {\rm cat}(f)$.
%Thus ${\rm cat}_L:[X,Y]_L\to (\bZ_{\geqq 0},\leqq ) $ is a poset map. Here $\op{cat}_L([f]_R):=\op{cat}(f).$ 

%\item When  $[g] \leqq_{LR} [f]$, 
% from
%$${\rm cat}(g)= {\rm cat}(t\circ f\circ s)\leq \min \{ {\rm cat}(t), {\rm cat}(f), {\rm cat}(s)\}\leqq {\rm cat}(f)$$
%for some maps $t:Y\to Y$ and $s:X\to X$ with $g\sim h\circ f\circ s$, we have  ${\rm cat}(g)\leqq {\rm cat}(f)$.
%Thus ${\rm cat}_{LR}:[X,Y]_{LR}\to (\bZ_{\geqq 0},\leqq ) $ is a poset map. Here $\op{cat}_{LR}([f]_R):=%\op{cat}(f).$ 
%\end{enumerate}

Namely we have the following commutative diagrams:
\begin{equation*}
\xymatrix{
[X,Y] \ar [r]^{\pi_R} \ar [d]_{\op{id}_{[X,Y]}} & [X,Y]_R \ar [d]^{\op{cat}_R} \\
[X,Y]  \ar[r]_{\op{cat}} & (\bZ_{\geqq 0},\leqq ), 
}
\xymatrix{
[X,Y]  \ar [r]^{\pi_L} \ar [d]_{\op{id}_{[X,Y]}} & [X,Y]_L \ar [d]^{\op{cat}_L} \\
[X,Y]  \ar[r]_{\op{cat}} & (\bZ_{\geqq 0},\leqq ),
}
\xymatrix{
[X,Y]  \ar [r]^{\pi_{LR}} \ar [d]_{\op{id}_{[X,Y]}} & [X,Y]_{LR} \ar [d]^{\op{cat}_{LR}} \\
[X,Y]  \ar[r]_{\op{cat}} & (\bZ_{\geqq 0},\leqq ) .
}
\end{equation*} \\

\end{ex}

%\begin{thm}
%If $[f]_L=[g]_L$, $G_n(Y,X;f)=G_n(Y,X;g)$.
%Furthermore if $[g]_L<[f]_L$,  there is a map $G_n(Y,X;g)\to G_n(Y,X;f)$.
%\end{thm}
%\begin{proof}
%Suppose that $f\simeq h\circ g$.
%Then  there is the map $h_*:G_n(Y,X;g)\subset G_n(Y,X;f)$ 
%given by $h_*(a)=h\circ a$. 
%\end{proof}

\begin{rem} Finally we remark that the referee pointed out that our machinery might be relevant to, for example, the following examples:
\begin{enumerate}
\item The theorem of Dehornoy (\cite{D1,D2,D3}) about natural orders on braid groups (e.g., see \cite{FGRRW}), which has given rise to considerable activity in low-dimensional topology, such as generalizations to knot group,
\item Elmendorf's theorem in equivariant homotopy theory, which describes $G$-equivariant homotopy types in terms of fixed-point spaces indexed by the orbit category of homogeneous spaces $G/H$ and $G$-maps between them (e.g., see \cite{May}): this yields natural stratifications of $G$-spaces,
\item Some related connections between homotopy theory and (equivariant) posets, e.g., such as a theorem saying that the category of ($G$-)posets admits a model structure that is Quillen equivalent to the standard model structure on the category of topological ($G$-)spaces\footnote{In \cite[page 83]{MZS} they write ``This implies that all of the algebraic topology of spaces can in principle be worked out in the category of posets. It can also be viewed as a bridge between the combinatorics of partial orders and algebraic topology."}  (e.g., see \cite{MZS}, \cite{R}, \cite{Thomason}).
\end{enumerate}
Furthermore the referee pointed out that he/she suspects that in the long run such poset structures will find an interpretation as part of Connes--Consani's recent theory \emph{``Homological algebra in characteristic one"} \cite{CC}.

In this paper we deal with only the homotopy set $[X,Y]$. However, if other things, e.g., the above examples and Connes--Consani's recent theory, are relevant to our machinery, then it would be quite interesting.\\
\end{rem}

\noindent
{\bf Acknowledgements:} We would like to thank Jim Stasheff for informing us of papers by K. Borsuk and P. Hilton, and also the referee for his/her pointing out those interesting/intriguing examples or works and Connes--Consani's recent theory.

\end{document}